\documentclass[preprint,12pt]{elsarticle}

\usepackage{amssymb}
\usepackage{amsmath, amsfonts, amssymb}
\usepackage{amsthm}
\newtheorem{theorem}{Theorem}[section]
\newtheorem{lemma}[theorem]{Lemma}
\newtheorem{proposition}[theorem]{Proposition}
\newtheorem{corollary}[theorem]{Corollary}
\theoremstyle{definition}
\newtheorem{definition}[theorem]{Definition}
\newtheorem{example}[theorem]{Example}

\theoremstyle{remark}
\newtheorem{remark}[theorem]{Remark}

\journal{}

\begin{document}

\begin{frontmatter}



\title{A generalization of order continuous operators}


\author{Mehrdad Bakhshi} 
\author{ Kazem Haghnejad Azar}

\address{Department  of  Mathematics  and  Applications, Faculty of Sciences, University of Mohaghegh Ardabili, Ardabil, Iran.\\
Email: mehrdad.bakhshi@uma.ac.ir \\	
Email: haghnejad@uma.ac.ir}

\begin{abstract}
Let $E$ be a sublattice of a vector lattice $F$.
A net $\{ x_\alpha \}_{\alpha \in \mathcal{A}}\subseteq E$ is said to be $ F $-order convergent to a vector $ x \in E$  (in symbols $ x_\alpha \xrightarrow{Fo} x $ in $E$), whenever there exists a net $ \{y_\beta\}_{\beta \in \mathcal{B}} $  in $F $ satisfying 
$ y_\beta\downarrow 0 $ in $F$ and for each $\beta$, there exists $\alpha_0$ such that $ \vert x_\alpha - x \vert \leq y_\beta $ whenever $ \alpha \geq \alpha_0 $.
In this manuscript, first we study some properties of $F$-order convergence nets and we extend some results  to the general cases. Let $E$ and $G$ be  sublattices of vector lattices $F$ and $H$ respectively.  We introduce $FH$-order continuous operators, that is,  an operator $T$ between two vector lattices $E$ and $G$  is said to be $FH$-order continuous, if $x_\alpha \xrightarrow{Fo} 0$ in $E$ implies $Tx_\alpha \xrightarrow{Ho} 0$ in $G$. We will study some properties of this new classification of operators and its relationships with order continuous operators. 
\end{abstract}




\end{frontmatter}


\section{Introduction and Preliminaries}
To state our result, we need to fix some notation and recall some definitions. A vector lattice E is an ordered vector space in which $\sup(x, y)$ exists for every
$x, y \in E$. A subspace $E$ of a vector lattice $F$ is said to be a sublattice if for every pair of elements $a, b$ of $E$ the supremum of $ a $ and $ b $ taken in $F$ belongs to $E$. A vector
lattice is said to be Dedekind complete (resp. $\sigma$-complete) if every nonempty subset (resp. countable subset) that is bounded from above has a supremum.\\
A sublattice $E$ of a vector lattice $F$ is said to be:
\begin{description}
	\item[order dense] if for every $0 < x \in F$ there exists $0 < y \in E $ such that $y \leq x$.
	\item[majorizing] if for every $x \in F$ there exists $y \in E $ such that $x \leq y$.
	\item[regular] if for every subset $A$ of $E$, $\inf{A}$ is the same in $F$ and in $E$ whenever $\inf{A}$ exists in $E$.
\end{description}
A Dedekind complete space $F$ is said to be a Dedekind completion of the vector lattice $E$ whenever $E$ is lattice isomorphic
to a majorizing order dense sublattice of $F$.
Recall that a non-zero element $ a \in E^+ $ is an \textbf{atom} iff the ideal $I_a$ consists only of the scalar multiples of $a$. 
Let $E$ be a vector lattice.  A net $\{ x_\alpha \}_{\alpha \in \mathcal{A}}\subseteq E$ is said to be order convergent (in short {\bf $o$-convergent}) to a vector $ x\in E$  (in symbols $ x_\alpha \xrightarrow{o} x $ ), whenever there exists a net  $ \{y_\beta\}_{\beta \in \mathcal{B}} $  in $E $ satisfying $ y_\beta\downarrow 0 $ and for each $\beta$ there exists $\alpha_0$ such that $ \vert x_\alpha - x \vert \leq y_\beta $ whenever $\alpha \geq \alpha_0$. Let $\{x_n\}$ be a sequence in a vector lattice $E$. Consider the sequence $\{ a_n \}$ of Ces$\acute{a}$ro means of $\{x_n\}$, defined by $a_n=\frac{1}{n}\sum_{k=1}^n x_k$.\\
Let $E$, $G$ be vector lattices. An operator $T:E\rightarrow G$ is said to be order bounded if it maps each order bounded subset of $E$ into order bounded subset of $G$. The collection of all order bounded operators from a vector lattice $E$ into a vector lattice $G$ will be denoted by ${L}_b(E,G)$. The vector space $E^\sim$ of all order bounded linear functionals on vector lattice      $E$ is called the {order dual} of $E$, i.e.,  $E^\sim =L_b(E ,\mathbb{R})$. Let  
$A$
be a  subset  of vector lattice        
$E$ and 
$Q_E$ be the natural mapping from $E$ into  $E^{\sim \sim}$.
If $Q_E(A)$
is  order  bounded  in 
$E^{\sim \sim}$, 
then  
$A$
is  said  to $b$-order  bounded  in
$E$. The concept of $b$-order bounded was first time itroduced by Alpay, Altin and Tonyali, see \cite{alpay2003property}. 
It  is  clear  that  every  order  bounded  subset  of
$E$
is  
$b$-order bounded. However, the  converse is    not  true  in general. For example,  the standard basis of $c_0$, 
$A= \{e_n \mid  n \in \mathbb{N}\}$ is
$b$-order bounded  in 
$c_0$
but  
$A$
is  not  order  bounded  in 
$c_0$.
A linear operator between two  vector lattices is order continuous (resp. $\sigma$-order continuous) if it maps order null nets (resp. sequences) to order null nets (resp. sequences). The collection of all order continuous (resp. $\sigma$-order continuous) linear operators from vector lattice $E$ into vector lattice $G$ will be denoted by $L_n(E,G)$ (resp. $L_c(E,G)$). 
For unexplained terminology and facts on  Banach lattices and positive operators, we refer the reader to \cite{1,aliprantis2006positive}.
\section{MAIN RESULTS}
In this section, $E$ is a sublattice of a vector lattice $F$. 
A net  $\{ x_\alpha \}_{\alpha \in \mathcal{A}}\subseteq E$ is said to be $ F $-order convergent (in short {\bf $Fo$-convergent}) to a vector $ x\in E$  (in symbols $ x_\alpha \xrightarrow{Fo} x $ ), whenever there exists a net $ \{y_\beta\}_{\beta \in \mathcal{B}} $  in $F $ satisfying $ y_\beta\downarrow 0 $ and for each $\beta$ there exists $\alpha_0$ such that $ \vert x_\alpha - x \vert \leq y_\beta $ whenever $\alpha \geq \alpha_0$. If $A\subseteq E$ is order bounded in $F$, we say that $A$ is $F$-order bounded, in case $F=E^{\sim\sim}$, we say that $A$ is $b$-order bounded.

It is clear that  if $E$ is regular in $F$, then every order convergence net (or order bounded set) in vector lattice $E$ is $F$-order convergent (or $F$-order bounded), but as following example the converse in general not holds. On the other hand, there is a sequence in $E$ that is order convergent  in $E$ and $F$, but is not $F$-order convergent in $E$.
\begin{example}\label{E0}
\begin{enumerate}
\item	Suppose that $E=c_{0}$ and  $F=\ell^\infty$.  The standard basis of $c_{0}$, $\{ e_n\}_{n=1}^{\infty}$  is not order convergence to zero, but  $\{ e_n\}_{n=1}^{\infty}$ is $\ell^\infty$-order convergent to zero. On the other hand $\{ e_n\}_{n=1}^{\infty}$ is not order bounded in $c_0$, but is $\ell^\infty$-order bounded in $c_0$.
\item Assume that $F$ is a set of real valued functions on $[0,1]$ of  form $f=g+h$ where $g$ is continuous and $h$ vanishes except at finitely many point. Let $E=C([0,1])$          and $f_n(t)=t^n$  where $t\in [0,1]$. It is obvious that $f_n\downarrow 0$ in $E$  and  $f_n\downarrow\chi_{\{1\}}$ in $F$, but $\{f_n\}$ is not $F$-order convergent. 
\end{enumerate}
\end{example}

It can easily be seen that a net in vector lattice $E$  have at most one $F $-order limit. The basic properties of $ Fo $-convergent are summarized in the next theorem.

\begin{theorem}
	Assume that the nets $ \{x_\alpha\} $ and  $ \{z_\gamma \} $ of a vector lattice $E$ satisfy $ x_\alpha \xrightarrow{Fo} x $ and $ z_\gamma \xrightarrow{Fo} z $. Then we have 
	\begin{enumerate}
		\item $ \vert x_\alpha \vert \xrightarrow{Fo} \vert x \vert $; $ x^{+}_\alpha \xrightarrow{Fo} x^{+} $ and $ x^{-}_\alpha \xrightarrow{Fo} x^{-} $.
		\item $ \lambda x_\alpha + \mu z_\gamma \xrightarrow{Fo} \lambda x + \mu z$ for all  $ \lambda , \mu \in \mathbb{R} $.
		\item $ x_\alpha \vee z_\gamma \xrightarrow{Fo} x \vee z $ and $ x_\alpha \wedge z_\gamma \xrightarrow{Fo} x \wedge z $.
		\item For each $y\in F$, if $ x_\alpha \leq y $ for all $ \alpha \geq \alpha_{0} $, then $ x \leq y $. 
		\item If  $ 0 \leq x_\alpha \leq  z_\alpha $ for all $\alpha$, then $ 0\leq x \leq z $.
		\item If $ P $ is order projection, then $ Px_\alpha \xrightarrow{Fo} Px $. 
	\end{enumerate} 
\end{theorem}
\begin{proof}
	These follow immediately by definition.
\end{proof}

\begin{theorem}\label{them1}
	Let $G$ be a sublattice of $E$ and $E$ be an ideal of $F$. Then the following statements hold:
	\begin{enumerate}
		\item  If $\{ x_\alpha \}_{\alpha \in \mathcal{A}} \subset G$ and $x_\alpha \xrightarrow{Eo}0 $  in $G$, then $x_\alpha \xrightarrow{Fo}0 $  in $G$.
		\item  If $\{ x_\alpha \}_{\alpha \in \mathcal{A}} \subset G$  is bounded in $E$ and $x_\alpha \xrightarrow{Fo}0 $  in $G$, then $x_\alpha \xrightarrow{Eo}0 $  in $G$.
	\end{enumerate}
\end{theorem}
\begin{proof}
	(1) Suppose $\{ x_\alpha \}_{\alpha \in \mathcal{A}} \subset G$ and $x_\alpha \xrightarrow{Eo}0 $ in $G$, there exists $\{y_\beta\}_{\beta  \in \mathcal{B}} \subset E $  with $y_\beta\downarrow 0$ in $E$ such that
	$$ \forall \beta  , \exists \alpha_0 \quad s.t \quad \forall \alpha \geq \alpha_0 : |x_\alpha | \leq y_\beta .$$
	We show that $y_\beta \downarrow 0$ in $F$. Let $0 \leqslant u \leqslant y_\beta $ for all $\beta$.
	Since $\{ y_\beta \} \subset E$ and $E$ is an ideal in $F$, it follows that $u \in E$ and hence $u=0$. Thus $y_\beta \downarrow 0$ in $F$. This means that $x_\alpha \xrightarrow{Fo}0 $ in $G$.\\
	(2) By assumption,  there exists $\{y_\beta\} \subset F$ satisfying, $ y_\beta \downarrow 0$ and for each $ \beta$  there exists $\alpha_0$ such that $|x_\alpha | \leq y_\beta $ whenever $ \alpha \geq \alpha_0$. Let $u \in E^+$ such that $|x_\alpha| \leq u $. Since $\{u \wedge y_\beta \} \subset E $  and $u \wedge y_\beta  \leq y_\beta  $, thus for each  $\beta$ there exists $\alpha_0$ that $|x_\alpha| \leq u \wedge y_\beta $ whenever $ \alpha \geq \alpha_0$. It follows that $x_\alpha \xrightarrow{Eo}0 $ in $G$.  
\end{proof}


\begin{corollary}\label{cor1}
	Suppose that $E$ is a Dedekind complete and an ideal of $F$. If $\{x_\alpha\}_{\alpha \in \mathcal{A}}$ is order bounded in $E$, then
	$$x_\alpha \xrightarrow{Fo} x \quad \textit{in} \quad E\quad\text{iff} \quad x_\alpha \xrightarrow{o} x \quad \textit{in} \quad E\quad.$$
\end{corollary}

As Example \ref{E0}, the condition of boundedness for nets in above corollary is necessary. Now the following example and part (2) of Example \ref{E0} show that the ideal condition is also necessary. 

\begin{example}\label{E02}
Assume that $E$ is a set of real valued continuous functions on $[0,1]$   except at finitely many point and $F$ is Lebesgue integrable real valued  functions on $[0,1]$. Obviously, $E$ is a sublattice in $F$, but is not  ideal in $F$. Let $I_1=(\frac{1}{3},\frac{2}{3})$, $I_2=(\frac{1}{9},\frac{2}{9})\cup (\frac{3}{9},\frac{6}{9})\cup (\frac{7}{9},\frac{8}{9}), ...$, the segments that we remove them for constructing of the Contor set $P$. It is obvious that $\chi_{I_n}\uparrow \chi_{P^c}$ in $F$, but $\{\chi_{I_n}\}$ is not $F$-order convergent in $E$.  
\end{example}

\begin{definition}
	\begin{enumerate}
		\item $E$ is said to be $F$-Dedekind (or $F$-order complete), if every nonempty $A\subseteq E$ that is bounded from above in $F$ has supermum in $E$. In case $F=E^{\sim\sim}$,  we say that $E$ is $b$-Dedekind complete.
		\item If each $F$-order bounded subset of $E$ is order bounded in $E$, then $E$ is said to have the property ($F$). In case $F=E^{\sim\sim}$,  we say that $E$ has property ($b$).
	\end{enumerate}
\end{definition}

\begin{remark}
Let $F$ be a  Dedekind complete $AM$-space with order unit $e$. If $E$ is a Dedekind complete closed in $F$ contain $e$, then $E$ has property $(F)$, see [\cite{Schaefer},  p.110].	We obvious that every majorizing sublattice $E$ of $F$ has the property ($F$). Since $E^\sim$ has $b$-property, $E^\sim$ is $b$-Dedekind complete.
	If $E$ is  $F$-Dedekind  complete, then $E$ is Dedekind complete. The converse of last assertion  in general not holds, of course $c_0$ is Dedekind complete, but is not     $\ell^\infty$-Dedekind complete. 
It is easy to show that a vector lattice $E$ has property ($F$) if and only if for each net $\{x_\alpha \}$ in $E$ with $x_\alpha\uparrow y$ for some $y\in F$, follows that $\{x_\alpha \}$ is bounded above in $E$.
\end{remark}
\begin{theorem}\label{theo2} 
	Suppose that $E$ is  $F$-Dedekind complete and an ideal of $F$. For each   net $\{x_\alpha \}_{\alpha \in \mathcal{A}}$ in $E$, we have the following assertions.
	\begin{enumerate}
		\item If $x_\alpha \uparrow  x$   or $x_\alpha \downarrow  x$ in $F$, then $x_\alpha \xrightarrow{Fo} x$ in $E$ .
		\item If $x_\alpha \uparrow$   \text{(resp.}~ $x_\alpha \downarrow$)  and  $x_\alpha \xrightarrow{Fo} x$ in $E$, then  $x_\alpha \uparrow  x$ (resp.  $x_\alpha \downarrow  x$) in $F$.
	\end{enumerate}
\end{theorem}
\begin{proof}
	(1) It is obvious by using Corollary \ref{cor1}.\\
	(2) Suppose $x_\alpha \uparrow  $ and  $x_\alpha \xrightarrow{Fo} x$ in $E$. There exists $\{y_\beta\} \subset F$ satisfying, $ y_\beta \downarrow 0$ and for fixed $ \beta_0$  there exists $\alpha_0$ such that $|x_\alpha | \leq y_{\beta_0} $ whenever $ \alpha \geq \alpha_0$. So 
\begin{align}
x-y_{\beta_0} < x_\alpha < x+ y_{\beta_0}, 
\end{align}	 
whenever $\alpha \geq \alpha_0$. 	Therefore,  $x_{\alpha} \leq y_{\beta_0} + x$ for all  $\alpha \geq \alpha_0$. Thus $\sup_{\alpha \geq \alpha_0} x_\alpha\leq y_{\beta_0}+x$, and so  $\sup_{\alpha \geq \alpha_0} x_\alpha\leq x$. It  follows that $\sup_{\alpha } x_\alpha\leq x$. Now, assume that $x_\alpha\leqslant u$ for some $u\in E$. By  (1), we have  $	x-y_{\beta_0} < x_\alpha\leqslant u$ for all $\alpha \geq \alpha_0$. It follows that $x\leqslant u$ and the proof follows. The second part has the similar argument.
\end{proof}

\begin{theorem}
	For each sequence $\{x_n\}_{n \in \mathbb{N}} $ in $E$ the following statements hold:
	\begin{enumerate}
		\item If $F$ has an order continuous norm and $x_n \xrightarrow{Fo} 0 $ in $E$, then there exists a subsequence $\{x_{n_k}\}$ such that $x_{n_k} \xrightarrow{o} 0 $ holds in $E$.
		\item  If $F$ is a Banach lattice and $\{x_n \}$ is norm convergent to $x\in E$, then there exists a subsequence $\{x_{n_k}\}$  such that  $x_{n_k} \xrightarrow{Fo} x$ holds in $E$.
		\item If $E$ is $F$-Dedekind complete and $x_n \xrightarrow{Fo}0 $   in $E$,  then $\{x_n\}_{n \in \mathbb{N}}$ is order bounded in $E$.
\end{enumerate}	
\end{theorem}
\begin{proof}
\begin{enumerate}
\item  There exists $\{y_m\}_{m \in \mathbb{N}}$ in $F$ satisfying, $y_m \downarrow 0 $ and for every $m$ there exists $n_0$ such that $|x_n| \leq y_m $ whenever $n \geq n_0$. By assumption $||y_m|| \longrightarrow 0 $, it follows that  $||x_n|| \longrightarrow 0 $. Pick subsequence $\{x_{n_k}\}$ of $\{x_n\}$ such that $|| x_{n_k} || < \frac{1}{2^k} $ for all $k$. Set $z_k= \sum_{i=k}^{\infty}|x_{n_k}| $. Clearly, $z_k \in E$ and for some $k_0$ we have $|x_{n_k}| \leq z_k \downarrow 0 $ whenever $n_k \geq k_0$ . This implies  that $x_{n_k} \xrightarrow{o} 0 $ holds in $E$.
\item By our hypothesis, there exists a subsequence $\{x_{n_k}\}$  such that $|| x_{n_k} -  x || \leq \frac{1}{k 2^{k}} $ for all $k$. Since $\sum_{k=1}^\infty k | x_{n_k} -  x |$ is norm convergence to some $u \in F$, then $ k | x_{n_k} -  x | \leq u$ for all $k$. Clearly, $\{\frac{1}{k}u\}$ is a sequence in $F$ such that $\frac{1}{k }u \downarrow 0 $ and $| x_{n_k} -  x | \leq \frac{1}{k}u$ and the proof is complete.
\item There exists a sequence $\{y_m\}_{m \in \mathbb{N}} $ in $F$ satisfying, $y_m \downarrow 0 $ and for every $m$ there exists $n_0$ such that $|x_n| \leq y_m $ whenever $n \geq n_0$. Fix $m\in \mathbb{N}$ such that  $|x_n| \leq y_m $ for all $n \geq n_0$. Put $z=\sup \{|x_1|,|x_2|,\ldots ,|x_{n_0-1}|,y_m\} $. Thus $|x_n| \leq z $ for all $n \in \mathbb{N}$, and so $z$ is an upper bound of $\{x_n\}$ in $F$. Since $E$ is $F$-Dedekind complete, it follows that $\{x_n\} $ is bounded in $E$.
\end{enumerate}
\end{proof}

\begin{remark}
	It is easy to see that for an order bounded net $\{x_\alpha\}$ in a Dedekind complete (order complete) vector lattice,
	$$  x_\alpha \xrightarrow{o} 0 \quad \text{in ~E~}~ \text{iff} \quad x=\inf_\alpha \sup_{\beta \geq \alpha} x_\beta= \sup_\alpha \inf_{\beta \geq \alpha} x_\beta \quad \text{in~E~} $$$$\text{iff} \quad x=\inf_\alpha \sup_{\beta \geq \alpha} |x_\beta - x |\text{~in~E}.$$
\end{remark}

The following fact is straightforward.

\begin{lemma}\label{Lemm1}
	Let $E$ be a sublattice of a Dedekind complete vector lattice $F$. Then
	$$  x_\alpha \xrightarrow{Fo} x \quad in \quad E \quad \text{iff} \quad x=\inf_\alpha \sup_{\beta \geq \alpha} x_\beta= \sup_\alpha \inf_{\beta \geq \alpha} x_\beta \  \text{in} \  F, $$
	for every order bounded net $\{x_\alpha\}$ in $E$.
\end{lemma}

A net $\{x_\alpha\}$ in $E$ is a $F$-order Cauchy, if the double net $\{(x_\alpha-x_\beta )\}_{(\alpha,\beta)}$ $F$-convergence in order to zero in $E$. The following proposition follows from the double equality of Lemma \ref{Lemm1} and the proof is straightforward. 
\begin{proposition}
Every $F$-order Cauchy net in an Dedekind complete vector lattice $E$ is order convergent.
\end{proposition}

For a vector lattice $E$, we write $E^\delta$ for its order ( or Dedekind) completion. Recall from \cite{a} Theorem 1.41 that $E^\delta$ is the unique ( up to a lattice isomorphism) order complete vector lattice that contains $E$ as
a majorizing and order dense sublattice. In particular, $E$ is regular sublattice of $E^\delta$.
\begin{theorem}\cite{folland}\label{theo101}
Let $E$ be a regular sublattice of a vector lattice $F$. Then
\begin{enumerate}
\item $E^\delta$ is a regular sublattice of $F^\delta$.
\item $x_\alpha \xrightarrow{o} 0$ in $E$ iff  $x_\alpha \xrightarrow{o} 0$ in $F$ for every order bounded net $\{x_\alpha\}$ in $E$.
\end{enumerate}	
\end{theorem}

\begin{theorem} 
	Suppose that $E$ is an order dense and majorizing sublattice of $F$. Then the both order convergence and $F$-order convergence are equivalent. 
\end{theorem}
\begin{proof}
	The forward implication follows by definition.  Now suppose that $x_\alpha \xrightarrow{Fo} 0$ in $E$. Then by definition there exists a net $ \{y_\beta\} $  in $F $ satisfying $ y_\beta\downarrow 0 $ and for each $\beta$ there exists $\alpha_0$ such that $ \vert x_\alpha - x \vert \leq y_\beta $ whenever $\alpha \geq \alpha_0$.
	Fix $\beta_0$ and set 
	$$ A=\{ u \in E \  : \  u \geq y_{\beta_0} \}. $$
	If $z \in F$ and $ 0 \leq z \leq x$ for all $x\in A$, then for every $\beta \geqslant \beta_0$,  we have 
	$ z \leq x$ whenever $x\in \{ u \in E \ : \ u \geq y_\beta \}$, 
	so that  by using Lemma 2.7 of \cite{folland}, we have  $z \leq y_\beta$  for all $\beta \geqslant \beta_0$. Hence $z=0$. On the other hand $\inf{A} = 0$ holds in $F$, and by regularity of $E$ in $F$, we have $\inf{A} = 0$ holds in $E$. Since $A$ is directed downwards, we may view $A$ as a decreasing net in $E$ and it is easy to see that this net dominates $\{x_\alpha-x\}$ in the sense of order convergence. Thus $\{x_\alpha\}$ is order convergent to $x$ in $E$.
\end{proof}
\begin{corollary}\label{cor10}
	For every net $\{x_\alpha\}$ in $E$,  $x_\alpha \xrightarrow{o} 0$ in $E$  iff  $x_\alpha \xrightarrow{E^\delta{o}} 0$ in $E$.
\end{corollary}
\begin{corollary}
	If $E$ is regular sublattice of $F$, then $x_\alpha \xrightarrow{E^\delta{o}} 0$ in $E$  iff $x_\alpha \xrightarrow{F^\delta{o}} 0$ in $E$  for every order bounded net $\{x_\alpha\}$  in $E^\delta$.
\end{corollary}
\begin{proof}
	From Corollary \ref{cor10} and Theorem   \ref{theo101}, it should be obvious.
\end{proof}

\begin{theorem}\label{theo11}
	Suppose that $E$ is a regular sublattice of a vector lattice $F$. Then   $x_\alpha \xrightarrow{o} 0$ iff $x_\alpha \xrightarrow{Fo} 0$ for  every order bounded net $\{x_\alpha\}$ in $E$.
\end{theorem}
\begin{proof}
	Since, for a bounded net, $F$-order convergence is equivalent to order convergence in $F$, thus by Theorem   \ref{theo101}  the result holds.
\end{proof}


\begin{theorem}
	Let $E$ be Dedekind $\sigma$-complete. Then the following statements hold:
	\begin{enumerate}
		\item  If  $\{x_n\}_{n \in \mathbb{N}}$ is a disjoint sequence in $E$. Then  $x_n \xrightarrow{Fo}$ $0$.
		\item  If $\{x_n\}_{n \in \mathbb{N}}$ is a sequence in $E$ and $x_n \xrightarrow{Fo} 0$, then Ces$\acute{a}$ro means of $\{x_n\}$ is $F$-order convergent to zero.
	\end{enumerate} 
\end{theorem}
\begin{proof}
\begin{enumerate}
\item Suppose  $\{x_n\}$ is a disjoint sequence in $E$. We claim that $y_n=\sup_{k\geq n}|x_k| \downarrow 0$ in $F$. Indeed, assume that $y\in F$ and   $\sup y_n \geq y \geq 0 $ for all $n \geq 1$. Then 
	$$0 \leq y \wedge |x_n| \leq (|x_n| \wedge \sup_{k \geq n+1} |x_k|)=\sup_{k \geq n+1}( |x_k|  \wedge |x_n|)=0,$$
	holds in $F$.
	Thus $ y \wedge |x_n|=0 $ holds in $F$ for all $n \geq 1$. It follows that 
	$$y=y \wedge (\sup_{n \geq 1} |x_n|)=\sup_{n \geq 1}(y \wedge |x_n|)=0 \ \ \ \text{for all} \ \ \ n \geq 1.$$
It follows that 	$|x_n|\leqslant y_n$ and $y_n\downarrow 0$ holds in $F$.
\item Since $F$ is  Dedekind complete, by \cite{1b},  we can find a sequence $\{y_n\}$  in $F$ such that $y_n \downarrow 0$ and $|x_n| \leq y_n $ for all $n$. Let $k_n$ be the
	integer part of $\ln(n)$. Then we have
	\begin{align}
	| \frac{1}{n}\sum_{k=1}^{n}x_i | & \leq \Big(\frac{1}{n}\sum_{k=1}^{n}|x_i |\Big)  \leq \Big(\frac{1}{n}\sum_{k=1}^{k_n+1}y_i  \Big) + \Big( \frac{1}{n}\sum_{k=k_n+2}^{n} y_i \Big)   \nonumber \\
	& \leq \Big(\frac{k_n+1}{n}y_1 + y_{k_n+2}\Big) \downarrow 0  . \nonumber 
	\end{align}
	It follows that $ | \frac{1}{n}\sum_{k=1}^{n}x_i | \xrightarrow{Fo} 0$ in $F$. Thus Ces$\acute{a}$ro means of $\{x_n\}$ is $F$-order convergent to zero.
\end{enumerate}
\end{proof}



\begin{theorem}
	Suppose that $E$ be a sublattice of a vector lattice $F$. Assume also $F$ is atomic and order continuous norm, and $\{x_n\}$ is an order bounded sequence in $F$. If $ x_n \rightarrow 0 $ then
	$ x_n \xrightarrow{Fo} 0 $.  
\end{theorem}
\begin{proof}
	It can be proven in same manner to Lemma 5.1 of \cite{mehrdad}. 
\end{proof}

\section{$FH$-order continuous operators}
In this section the basic properties of $FH$-order continuous operators will be
studied. In the following, $E$ and $G$ are vector sublattices of vector lattices $F$ and $H$, respectively.  Let $\{x_n\}\subseteq E$ and $x\in E$. The notation $ x_\alpha\downarrow_{F} x $ means that $ x_\alpha\downarrow $ and $ \inf\{x_\alpha\} = x $ holds in  $  F $. 
Assume that $F$ is a set of real valued functions on $[0,1]$ of  form $f=g+h$ where $g$ is continuous and $h$ vanishes except at finitely many point. Let $E=C([0,1])$,  $\{f_n\}$ be a decreasing sequence in $E^+$ such that $f_n(\frac{1}{2})=1$ for all  $n\in \mathbb{N}$ and $f_n(t)\rightarrow 0$ for every $t\neq \frac{1}{2}$.         It is clear that $f_n\downarrow 0$ in $E$,  but   $f_n\downarrow_F 0$ not holds.


\begin{definition}\label{def2}
	An operator $T : E \longrightarrow G$ between two vector lattices is said to be:\\
	(a) $FH$-order continuous, if $x_\alpha \xrightarrow{Fo}0$ in $E$ implies $Tx_\alpha \xrightarrow{Ho} 0$ in $G$.\\
	(b) $FH$-$\sigma$-order continuous, if $x_n\xrightarrow{Fo}0$ in $E$ implies $Tx_n \xrightarrow{Ho} 0$ in $G$.
\end{definition}

The collection of all $FH$-order continuous operators will be denoted by $L_{FHn}(E,G)$,  that is
$$L_{FHn}(E,G)=\{T \in L(E,G) :    T  \; \text{is} \;  FH \text{-order continuous} \}.$$
Similarly, $L_{FHc}(E,G)$ will denoted by the collection of all order bounded operators 
from $E$ to $G$ that are $FH$-$\sigma$-order continuous. That is,
$$L_{FHc}(E,G)=\{T \in L(E,G) :    T  \; \text{is} \;  FH\text{-}\sigma\text{-order continuous} \}.$$


\begin{lemma}\label{theo3}
	Let  $E$ and $G$ are $F$-Dedekind complete and $H$-Dedekind complete, respectively. Then we have
	the following assertions.
	\begin{enumerate}
		\item  $0 \leq T \in L_{FHn}(E,G)$ if and only if for each net $\{x_\alpha\}$ in $E$, $x_\alpha \downarrow_F 0$ implies $Tx_\alpha \downarrow_H 0$.
		\item In addition, if $E$ and $G$ are ideals in $F$ and $H$, respectively, then $L_{FHc}(E,G)=L_{c}(E,G)$. Moreover, the $FH$-order continuous operator $T$ is an order bounded.
	\end{enumerate}	
\end{lemma}
\begin{proof}
\begin{enumerate}
\item Assume that $0 \leq T$ and $\{x_\alpha\}$ is a net in $E$ such that $x_\alpha \xrightarrow{Fo} 0$.   It follows that there exists a net $\{y_\beta\}_{\beta \in \mathcal{B}}$ in $F$ satisfying, $ y_\beta \downarrow 0 $ and for each $\beta$ there exists $ \alpha_0 $ such that $|x_\alpha| \leq y_\beta $ whenever $\alpha \geq \alpha_0 $. Set $ z_\alpha = \bigvee_{\lambda \geq \alpha} |x_\lambda |$. Then we have $|x_\alpha| \leq |z_\alpha| $ for all $\alpha$ and $z_\alpha \downarrow_F 0$, and so by our assumption $Tz_\alpha \downarrow_H 0$. Since $|Tx_\alpha | \leq T|x_\alpha| \leq Tz_\alpha $ whenever $ \alpha \geq \alpha_0$, it follows that $Tx_\alpha \xrightarrow{Ho} 0$.\\
	Conversely, suppose that  $0 \leq T \in L_{FHn}(E,G)$ and $x_\alpha \downarrow_F 0$. It follows that $x_\alpha \xrightarrow{Fo} 0$ and so $Tx_\alpha \xrightarrow{Ho} 0$. Then there exists a net $\{z_\beta\}$ in $H$ satisfying, $z_\beta \downarrow 0$ and for each $\beta$ there exists $\alpha_0$ such that $|Tx_\alpha| \leq z_\beta $ whenever $ \alpha \geq \alpha_0$, which shows that $Tx_\alpha \downarrow_H 0$ and the proof is follows.
\item  Suppose that $T \in L_{FHc}(E,G) $ and $x_n \xrightarrow{o} 0 $. By using Corollary \ref{cor1}, we have  $x_n \xrightarrow{Fo} 0 $ and by our assumption, we have $Tx_n \xrightarrow{Ho} 0 $. Using Corollary \ref{cor1} again, we have $Tx_n \xrightarrow{o} 0 $ and so $T \in L_{c}(E,G) $. The converse is proved in the same manner. For the last part, let  $0 \leq T \in L_{FHn}(E,G)$ and $x_0 \in E^+$. If we consider the order interval $[0,x_0]$ as a net $\{x_\alpha\}$ where $x_\alpha = \alpha $ for each $\alpha \in [0,x_0] $, then $x_\alpha \uparrow x_0 $ holds in $F$. So by using Theorem \ref{theo2}, we have $x_\alpha \xrightarrow{Fo} x_0 $ and therefore  $Tx_\alpha \uparrow Tx_0 $ holds in $H$. Thus $T$ is order bounded.	
\end{enumerate}
\end{proof}
\begin{theorem}
Let $H$ be Dedekind complete. 	If  $ T\in L_{n}(F,H)$, then $T\in L_{FHn}(E,H)$.
\end{theorem}
\begin{proof}
By Lemma \ref{theo3}, since $T$ is order bounded, follows that $T=T^+-T^-$. Thus without loss of generality, we assume that $T$ is a positive operator.	Suppose that $\{x_\alpha\}_{\alpha \in \mathcal{A}}$ is a net in $E$ which is $F$-order convergent to zero, then there exists a net $\{y_\beta\}_{\beta \in \mathcal{B}}$ in $F$ satisfying, $ y_\beta \downarrow 0 $ and for each $\beta$ there exists $ \alpha_0 $ such that $|x_\alpha| \leq y_\beta $ holds whenever $\alpha \geq \alpha_0 $. Since $T$ is a positive operator, we have $|Tx_\alpha| \leq T|x_\alpha|\leq T(y_\beta) $ whenever $ \alpha \geq \alpha_0$. By assumption we have $T(y_\beta)\downarrow 0$ in $H$  and so the proof follows.
\end{proof}
An other application of the preceding lemma yields the following corollary, in which the techniques of this corollary are similar argument like as Theorem 1.56 \cite{aliprantis2006positive}, so we omit its proof.
\begin{corollary}\label{theo10}
	Let  $E$ and $G$ are $F$-Dedekind complete and $H$-Dedekind complete, respectively. Then the following assertions are equivalent.
\end{corollary}
\begin{enumerate}
	\item  $ T \in L_{FHn}(E,G)$.
	\item  $x_\alpha \downarrow_F 0$ implies $Tx_\alpha \downarrow_H 0$.
	\item  $x_\alpha \downarrow_F 0$ implies $\inf_H|Tx_\alpha|=0$.
	\item $ T^+, T^- \ and  \ |T|\ \ belong \ to \ L_{FHn}(E,G)$. 
\end{enumerate}

The next result  presents a useful sufficient condition for a set to be order bounded in two vector lattices.
\begin{theorem}
	Let $I$ be a sublattice of $E$ and $E$ be $F$-Dedekind complete. Then subset $A$ of $I$ is $E$-order bounded if and only if $F$-order bounded.
\end{theorem}

	An operator $T:E\rightarrow G$ is said to be  $FH$-order bounded if it maps each $F$-order bounded subset of $E$ into $H$-order bounded subset of $G$. The collection of all  $FH$-order bounded operators from a vector lattice $E$ into a vector lattice $G$ will be denoted by ${L}_{FHb}(E,G)$.

An order bounded operator between two vector lattices $E$ and $G$ is also $FH$-order bounded. However as following example, the converse need not be true.
\begin{example}
	Let $T : L^1 [0,1] \longrightarrow c_0 $ be defined by 
	$$ T(f)=(\int_{0}^{1}f(x) \sin(x)dx,\int_{0}^{1}f(x) \sin(2x)dx,\ldots).$$
	Then $T$ is a $L^\infty[0,1]\ell^\infty$-order bounded but is not an order bounded operator.
\end{example}

\begin{theorem}\label{theo5}
	For two vector lattices $E$ and $F$, we have the following:
	\begin{enumerate}
	    \item $L_{FHc}(E,G)\subseteq L_{FHb}(E,G)$. 
		\item If $E$ has property (F), then $L_b(E,G) \subseteq {L}_{FHb}(E,G)$. 
		\item If $G$ has property (H), then ${L}_{FHb}(E,G) \subseteq L_b(E,G)$.
		\item If $E$ and $G$ have property (F) and (H), respectively, then $L_b(E,G) = {L}_{FHb}(E,G)$. 
	\end{enumerate}
	
\end{theorem}
\begin{proof}
\begin{enumerate}
\item The proof is clear.
\item Suppose that $T \in L_{b}(E,G)$ and $ A \subset E$ is an $F$-order bounded subset of $E$. From our hypothesis, $A$ is an order bounded subset of $E$ and $T(A)$ is an order bounded subset of $G$. Therefore $T(A)$ is an $H$-order bounded subset of $G$ and hence $T \in L_{FHb}(E,G)$.	
\item Assume that $T \in L_{FHb}(E,G)$ and $ A \subset E$ is an order bounded subset of $E$. Then $A$ is an $F$-order bounded subset of $E$ and from our hypothesis, $T(A)$ is an $H$-order bounded subset of $G$. Since $G$ has property (H), $T(A)$ is an order bounded subset of $G$ and therefore $T \in L_b(E,G)$.
\item It is obvious by (1) and (2).
\end{enumerate}
\end{proof}
\begin{corollary}\label{cor13}
	Let $E$ and $G$ be $F$-Dedekind complete and $H$-Dedekind complete, respectively, then $L_b(E,G)={L}_{FHb}(E,G)$.
\end{corollary}

\begin{corollary}
	Let  $E$ be an $F$-Dedekind complete ideal of $F$. Assume also $G$ is an $H$-Dedekind complete ideal of $H$. Then  $L_{FHn}(E,G)$ and $L_{FHc}(E,G)$ are both bands of $L_{FHb}(E,G)$. 
\end{corollary}
\begin{proof}
       Corollary \ref{cor13} and part 2 of Lemma \ref{theo3}, show that  $L_{FHn}(E,G)$ and $L_{FHc}(E,G)$ are both subspaces of $L_{FHb}(E,G)$ and the rest of the proof has a similar argument like as Theorem 1.57 \cite{aliprantis2006positive}. 
\end{proof}





\end{document}